\pdfoutput = 1
\documentclass[11pt]{article}
\usepackage{enumerate}
\usepackage[OT1]{fontenc}
\usepackage[usenames]{color}
\usepackage{smile}
\usepackage[colorlinks,
            linkcolor=red,
            anchorcolor=blue,
            citecolor=blue
            ]{hyperref}
\usepackage{mathrsfs}
\usepackage{fullpage}
\usepackage{hyperref}
\usepackage[protrusion=true, expansion=true]{microtype}
\usepackage{float}
\usepackage{subfigure}
\usepackage{amsfonts,amsmath,amssymb,amsthm,url,xspace}
\usepackage{tikz}
\usepackage{verbatim}
\usetikzlibrary{arrows,shapes}
\usepackage{mathtools}
\usepackage{authblk}
\usepackage[bottom]{footmisc}
\usepackage{enumitem}
\usepackage{mathtools}
\usepackage{lscape}
\usepackage{caption}

\usepackage{algorithm}
\usepackage{algorithmicx,algpseudocode}

\ifx\counterwithout\undefined\usepackage{chngcntr}\fi
\counterwithout{equation}{section}

\allowdisplaybreaks[1]

\def\L{{\mathcal L}}

\def\0{{\boldsymbol 0}}

\def\b1{{\boldsymbol{1}}}
\def\bx{{\boldsymbol{x}}}

\def \b0{{\boldsymbol{0}}}

\mathtoolsset{showonlyrefs=true}

\usepackage{xargs}
\usepackage[colorinlistoftodos,prependcaption,textsize=tiny]{todonotes}
\newcommandx{\unsure}[2][1=]{\todo[linecolor=red,backgroundcolor=red!25,bordercolor=red,#1]{#2}}
\newcommandx{\change}[2][1=]{\todo[linecolor=blue,backgroundcolor=blue!25,bordercolor=blue,#1]{#2}}
\newcommandx{\info}[2][1=]{\todo[linecolor=OliveGreen,backgroundcolor=OliveGreen!25,bordercolor=OliveGreen,#1]{#2}}
\newcommandx{\improvement}[2][1=]{\todo[linecolor=Plum,backgroundcolor=Plum!25,bordercolor=Plum,#1]{#2}}

\usepackage{alphalph}

\allowdisplaybreaks[1]

\begin{document}

\title{Trust-Region Stochastic Optimization with Variance Reduction Technique}

\author[1]{Xinshou Zheng}

\affil[1]{Boston University }

\date{}

\maketitle

\section{Introduction}\label{sec:1}

We propose a novel algorithm, TR-SVR, for solving unconstrained stochastic optimization problems. This method builds on the trust-region framework, which effectively balances local and global exploration in optimization tasks. TR-SVR incorporates variance reduction techniques to improve both computational efficiency and stability when addressing stochastic objective functions. The algorithm applies a sequential quadratic programming (SQP) approach within the trust-region framework, solving each subproblem approximately using variance-reduced gradient estimators. This integration ensures a robust convergence mechanism while maintaining efficiency, making TR-SVR particularly suitable for large-scale stochastic optimization challenges.

Unlike traditional SQP methods typically designed for deterministic or constrained optimization problems, TR-SVR is specifically tailored to address unconstrained stochastic settings. This makes it highly applicable to large-scale machine learning and data-driven tasks, where efficiency and scalability are crucial. 

The trust-region mechanism in TR-SVR dynamically adjusts the step size by defining a region where the quadratic approximation of the objective function is reliable. This ensures that the algorithm progresses steadily while avoiding excessively large or overly cautious updates. Simultaneously, variance reduction techniques inspired by methods like Stochastic Variance Reduced Gradient (SVRG) significantly reduce the noise in stochastic gradient estimates, thereby improving both the stability and accuracy of the optimization process. 

By iteratively refining the solution and adaptively modifying the trust-region radius based on the quality of gradient estimates and the current solution, TR-SVR achieves faster convergence rates and enhanced robustness, even in noisy environments typical of stochastic optimization problems.

\section{Literature Review}\label{sec:2}

Stochastic optimization is a rapidly growing field due to its pivotal role in applications like machine learning, signal processing, and control systems. Among its foundational techniques is Stochastic Gradient Descent (SGD), introduced by Robbins and Monro in the mid-20th century \cite{robbins1951stochastic}. SGD has gained popularity for its simplicity and scalability, particularly in large-scale optimization tasks. However, its high variance in gradient estimates often leads to slow convergence and instability. To mitigate this issue, advanced variance reduction techniques such as Stochastic Variance Reduced Gradient (SVRG) \cite{johnson2013accelerating}, Stochastic Average Gradient (SAGA) \cite{defazio2014saga}, and Stochastic Recursive Gradient Algorithm (SARAH) \cite{nguyen2017sarah} have been developed. These methods refine the gradient estimates through mechanisms like reference points or control variates, significantly improving performance.

In parallel, trust-region methods have emerged as robust tools for handling non-convex optimization problems, ensuring global convergence through adaptive step-size control \cite{conn2000trust}. These methods dynamically adjust a region around the iterate within which a quadratic model of the objective is trusted, making them highly effective in deterministic settings, as detailed in foundational works like Nocedal and Wright's "Numerical Optimization" \cite{nocedal2006numerical}. Efforts to extend trust-region methods to stochastic domains have gained momentum. Curtis et al. \cite{curtis2019fully} introduced a fully stochastic second-order trust-region method leveraging stochastic Hessian approximations. More recently, Fang et al. \cite{Fang2024Fully} proposed a stochastic trust-region sequential quadratic programming (TR-SQP) method tailored for equality-constrained problems, offering robust theoretical guarantees and practical implementations. These developments highlight the growing synergy between stochastic optimization and trust-region frameworks, paving the way for tackling increasingly complex optimization challenges.

Sequential Quadratic Programming (SQP) methods are widely recognized for their effectiveness in solving constrained optimization problems \cite{boggs1995sequential}. These methods operate by solving a sequence of quadratic subproblems that locally approximate the original nonlinear problem, gradually refining the solution at each iteration. While initially designed for deterministic optimization, recent advancements have extended SQP to stochastic settings. Notably, Berahas et al. \cite{berahas2022accelerating} proposed a stochastic SQP framework that integrates variance reduction techniques, significantly enhancing convergence rates for equality-constrained problems compared to traditional first-order approaches.

Despite these promising developments, existing stochastic SQP methods are predominantly tailored for constrained problems or require intricate adjustments to address challenges like non-convexity and noisy gradients. Our proposed TR-SVR algorithm addresses this gap by offering a novel framework for unconstrained stochastic optimization. Integrating trust-region principles with advanced variance reduction strategies, TR-SVR extends the applicability of stochastic SQP methods while preserving computational efficiency and delivering robust theoretical guarantees.
\section{Algorithm}\label{sec:3}

\subsection{Problem Description}\label{sec:3.1}

We consider the unconstrained stochastic optimization problem of the form:

\[
    \min_{\boldsymbol{x} \in \mathbb{R}^{d}} f(\boldsymbol{x}) = \frac{1}{N} \sum_{i=1}^{N} f_i(\boldsymbol{x}),
\]
where \( f_i(\boldsymbol{x}) \) represents individual stochastic functions, and \( f(\boldsymbol{x}) \) is the overall objective. This formulation is common in large-scale machine learning applications, where the objective function is typically a sum of loss functions over a dataset. The main challenge in solving such problems arises from the stochastic nature of the gradients, which can introduce high variance and slow down convergence.

Our goal is to develop an efficient algorithm that can handle large-scale problems by reducing the variance in gradient estimates while maintaining computational efficiency. To this end, we propose a novel algorithm, TR-SVR, which integrates variance reduction techniques into a trust-region-based sequential quadratic programming (SQP) framework. Unlike traditional methods that rely on full gradients or line-search techniques, our approach uses mini-batch gradient estimates and dynamically adjusts the trust-region radius to ensure robust performance in noisy environments.

The TR-SVR algorithm operates in two loops: an outer loop indexed by \(k\), and an inner loop indexed by \(s\). In each iteration, a quadratic approximation of the objective function is constructed using variance-reduced gradient estimates. A trust-region subproblem is then solved to update the current iterate. The trust-region radius is adjusted dynamically based on the quality of the solution at each step, ensuring that the algorithm takes appropriately sized steps to balance exploration and exploitation.

\subsection{Algorithm Description}\label{sec:3.2}

The TR-SVR algorithm is presented below. It combines trust-region principles with variance reduction techniques to solve unconstrained stochastic optimization problems efficiently.

\begin{algorithm}
    \caption{TR-SVR Algorithm}
    \begin{algorithmic}[1]
        \State \textbf{Input:} Initial iterate \(x_0 \in \mathbb{R}^d\), initial trust-region radius \(\Delta_0 > 0\), batch size \(b \in [1, N]\), maximum inner iterations \(S > 0\), parameters \(\eta_1, \eta_2 > 0\), and scaling factor \(\alpha > 0\).
        
        \For{$k = 0, 1, 2, \dots$}
            \State Set \(x_{k,0} = x_{k-1,S}\) (if \(k > 0\)) or initialize \(x_{k,0}\).
            \State Compute full gradient at \(x_{k,0}\): 
            \[
                g_{k,0} = \nabla f(x_{k,0}) = \frac{1}{N} \sum_{i=1}^{N} \nabla f_i(x_{k,0}).
            \]
            
            \For{$s = 0$ to $S-1$}
                \State Select a mini-batch \(I_{k,s} \subset [N]\) of size \(b\).
                \State Compute mini-batch gradient estimate:
                \[
                    \tilde{g}_{k,s} = \frac{1}{b} \sum_{i \in I_{k,s}} \nabla f_i(x_{k,s}).
                \]
                
                \State Compute variance-reduced gradient:
                \[
                    \bar{g}_{k,s} = \tilde{g}_{k,s} - (\tilde{g}_{k,0} - g_{k,0}),
                \]
                where \(g_{k,0}\) is the full gradient at \(x_{k,0}\).
                
                \State Solve the trust-region subproblem:
                \[
                    m(\Delta x) = g_{k,s}^{T} (\Delta x) + \frac{1}{2} (\Delta x)^T H_{k,s} (\Delta x),
                \]
                subject to 
                \[
                    ||\Delta x||_2 \leq \Delta_{k,s},
                \]
                where \(H_{k,s}\) is an approximation of the Hessian matrix.
                
                \State Update the iterate:
                \[
                    x_{k,s+1} = x_{k,s} + (\Delta x)_{k,s}.
                \]
                
                \State Adjust trust-region radius:
                If the reduction in objective function meets certain criteria (e.g., sufficient decrease), increase or decrease the trust-region radius as follows:
                \[
                    \Delta_{k,s+1} = 
                    \begin{cases}
                        {\eta_1}\alpha ||g_{k,s}|| & ||g_{k,s}|| < 1/\eta_1,\\
                        {\alpha}, & 1/\eta_1 < ||g_{k,s}|| < 1/\eta_2,\\
                        {\eta_2}\alpha ||g_{k,s}|| & ||g_{k,s}|| > 1/\eta_2.
                    \end{cases}
                \]
                
            \EndFor
            
            Update outer loop iterate: 
            Set \(x_{k+1,0} = x_{k,S}\).
        \EndFor
    \end{algorithmic}
\end{algorithm}

The TR-SVR algorithm iteratively refines both the solution and the trust-region radius using variance-reduced gradient estimates. The use of mini-batches ensures computational efficiency, particularly in large-scale settings, while the dynamic adjustment of the trust-region radius maintains stability, even in noisy environments. To further improve efficiency, the quadratic subproblems are solved approximately using Hessian approximations, thereby avoiding the need for expensive second-order computations.
\section{Assumptions}
\begin{assumption}
The objective function $f: \mathbb{R}^n \to \mathbb{R}$ is twice continuously differentiable and bounded below by a scalar $f_{\inf} := \inf_{x \in \mathbb{R}^n} f(x) > -\infty$. The gradient $\nabla f(x)$ is Lipschitz continuous with constant $L_g > 0$, and the Hessian matrix $\nabla^2 f(x)$ is uniformly bounded, i.e., $\|\nabla^2 f(x)\|_2 \leq L_H$ for all $x \in \mathbb{R}^n$. Additionally, the Hessian approximation $H_{k,s}$ satisfies $\|H_{k,s}\| \leq K_H$ for some constant $K_H > 0$ and for all iterations $k,s$.
\end{assumption}

\begin{assumption}
    The gradient approximation $\tilde{g}_{k,s}$ is an unbiased estimator of the true gradient of the objective function, i.e., $\mathbb{E}_{k,s}[\tilde{g}_{k,s}] = g_{k,s} = \nabla f(x_{k,s})$, where the expectation is conditioned on the event that the algorithm has reached $x_{k,s}$. The variance of the stochastic gradient is bounded, i.e., $\mathbb{E}_{k,s}[\|\tilde{g}_{k,s} - g_{k,s}\|^2] \leq \sigma_g^2$.
\end{assumption}

\begin{assumption}
    The variance-reduced gradient estimate satisfies a variance bound such that:
\[
    \mathbb{E}_{k,s}\left[\|\bar{g}_{k,s} - g_{k,s}\|^2\right] \leq \frac{L^2}{b}\|x_{k,s} - x_{k,0}\|^2,
\]
where $b$ is the mini-batch size and $L$ is a Lipschitz constant.
\end{assumption}

\begin{assumption}
    The iterates $x_{k,s}$ are contained within a compact convex set $\mathcal{X} \subseteq \mathbb{R}^n$. The objective function $f(x)$, its gradient $\nabla f(x)$, and its Hessian matrix are bounded over this set. Each component function $f_i(x)$ is continuously differentiable, and its gradient is Lipschitz continuous with constant $L > 0$.
\end{assumption}

\section{Convergence Analysis}\label{sec:5}

We begin by analyzing the trust-region subproblem and establishing key properties of the variance-reduced gradient estimates.

\subsection{Trust-Region Subproblem Properties}

At each iterate \(x_{k,s}\), the trust-region subproblem is formulated as:
\[
    \min_{{\Delta x_{k,s}}} m_1(\Delta x_{k,s}) = \bar{g}_{k,s}^T \Delta x_{k,s} + \frac{1}{2} \Delta x_{k,s}^T H_{k,s} \Delta x_{k,s}\quad \text{s.t. }\|\Delta x_{k,s}\| \leq \Delta_{k,s}
\]

For this subproblem, we only require the Cauchy decrease condition rather than an exact solution:
\[
    m_1(\Delta x_{k,s}) - m_1(0) \leq -\|\bar{g}_{k,s}\| \Delta_{k,s} + \frac{1}{2}\|H_{k,s}\| \Delta_{k,s}^2
\]

\subsection{Variance Reduction Properties}

\begin{lemma}[Variance Bound]\label{lem:5.1}
Let \(\bar{g}_{k,s}\) be computed as in Algorithm 1. Then for all \([k,s] \in \mathbb{N} \times S\), we have:
\[
    \mathbb{E}_{k,s}[\|\bar{g}_{k,s} - g(x_{k,s})\|^2] \leq \mu_{k,s}
\]
where \(\mu_{k,s} = \frac{L^2}{b}\|x_{k,s} - x_{k,0}\|^2\).
\end{lemma}

\begin{proof}
Let us denote:
\[
    J_{k,s} = \frac{1}{b} \sum (\nabla f_i(x_{k,s}) - \nabla f_i(x_{k,0}))
\]

By Assumption 4.2, we have that \(\mathbb{E}_{k,s}[J_{k,s}] = g(x_{k,s}) - g(x_{k,0})\).

Using the fact that \(\mathbb{E}[\|z-\mathbb{E}[z]\|^2] \leq \mathbb{E}[\|z\|^2]\) (from Assumption 4.3) and the property that for independent mean-zero random variables \(z_1, z_2, ..., z_n\):
\[
    \mathbb{E}[\|z_1 + z_2 + ... + z_n\|^2] = \mathbb{E}[\|z_1\|^2 + \|z_2\|^2 + ... + \|z_n\|^2]
\]

We can derive:
\begin{align*}
    \mathbb{E}_{k,s}[\|\bar{g}_{k,s} - g(x_{k,s})\|^2] &= \mathbb{E}_{k,s}[\|J_{k,s} + g(x_{k,0}) - g(x_{k,s})\|^2] \\
    &= \mathbb{E}_{k,s}[\|J_{k,s} - \mathbb{E}[J_{k,s}]\|^2] \\
    &= \frac{1}{b^2}\mathbb{E}_{k,s}[\|\sum(\nabla f_i(x_{k,s}) - \nabla f_i(x_{k,0}) - \mathbb{E}[J_{k,s}])\|^2] \\
    &= \frac{1}{b^2}\mathbb{E}_{k,s}[\sum\|\nabla f_i(x_{k,s}) - \nabla f_i(x_{k,0}) - \mathbb{E}[J_{k,s}]\|^2] \\
    &\leq \frac{1}{b^2}\mathbb{E}_{k,s}[\sum\|\nabla f_i(x_{k,s}) - \nabla f_i(x_{k,0})\|^2] \\
    &\leq \frac{L^2}{b}\|x_{k,s} - x_{k,0}\|^2
\end{align*}

The last inequality follows from Assumption 4.4, which ensures the Lipschitz continuity of the component gradients with constant \(L\).
\end{proof}

\subsection{One-Step Decrease Properties}

\begin{lemma}[One-Step Decrease]\label{lem:5.2}
For any iteration \((k,s)\), we have:
\[
    f(x_{k,s+1}) - f(x_{k,s}) \leq -\|\bar{g}_{k,s}\|\Delta_{k,s} + \frac{1}{2}\|H_{k,s}\|\Delta_{k,s}^2 + \|g(x_{k,s}) - \bar{g}_{k,s}\|\Delta_{k,s} + \frac{1}{2}(L_{\nabla f} + \|H_{k,s}\|)\Delta_{k,s}^2
\]
\end{lemma}

\begin{proof}
By Assumption 4.1, which ensures twice continuous differentiability and Lipschitz continuity of the gradient, we can write:
\[
    f(x_{k,s+1}) \leq f(x_{k,s}) + g(x_{k,s})^T\Delta x_{k,s} + \frac{1}{2}L_{\nabla f}\|\Delta x_{k,s}\|^2
\]

Using the trust-region subproblem formulation and the Cauchy decrease condition:
\[
    f(x_{k,s+1}) - f(x_{k,s}) - \bar{g}_{k,s}^T\Delta x_{k,s} - \frac{1}{2}\Delta x_{k,s}^T H_{k,s}\Delta x_{k,s}
\]
\[
    \leq g(x_{k,s})^T\Delta x_{k,s} + \frac{1}{2}L_{\nabla f}\|\Delta x_{k,s}\|^2 - \bar{g}_{k,s}^T\Delta x_{k,s} - \frac{1}{2}\Delta x_{k,s}^T H_{k,s}\Delta x_{k,s}
\]
\[
    = (g(x_{k,s}) - \bar{g}_{k,s})^T\Delta x_{k,s} + \frac{1}{2}L_{\nabla f}\|\Delta x_{k,s}\|^2 - \frac{1}{2}\Delta x_{k,s}^T H_{k,s}\Delta x_{k,s}
\]

Using the Cauchy-Schwarz inequality and the fact that \(\|\Delta x_{k,s}\| \leq \Delta_{k,s}\):
\[
    \leq \|g(x_{k,s}) - \bar{g}_{k,s}\|\|\Delta x_{k,s}\| + \frac{1}{2}(L_{\nabla f} + \|H_{k,s}\|)\|\Delta x_{k,s}\|^2
\]
\[
    \leq \|g(x_{k,s}) - \bar{g}_{k,s}\|\Delta_{k,s} + \frac{1}{2}(L_{\nabla f} + \|H_{k,s}\|)\Delta_{k,s}^2
\]

Therefore, combining with the trust-region subproblem solution property:
\[
    f(x_{k,s+1}) - f(x_{k,s}) \leq -\|\bar{g}_{k,s}\|\Delta_{k,s} + \frac{1}{2}\|H_{k,s}\|\Delta_{k,s}^2 + \|g(x_{k,s}) - \bar{g}_{k,s}\|\Delta_{k,s} + \frac{1}{2}(L_{\nabla f} + \|H_{k,s}\|)\Delta_{k,s}^2
\]

This result relies on Assumptions 4.1 (Lipschitz continuity), and 4.4 (boundedness of iterates).
\end{proof}

\subsection{Expected Decrease Properties}

\begin{lemma}[Expected Decrease]\label{lem:5.3}
When \(\alpha \leq \frac{1}{2(L_{\nabla f} + 2K_H)}\), we have:
\[
    \mathbb{E}_{k,s}[f(x_{k,s+1})] - f(x_{k,s}) \leq -\frac{1}{2}\alpha\|g(x_{k,s})\|^2 + \frac{1}{2}(L_{\nabla f} + 2K_H)\alpha^2\mathbb{E}_{k,s}[\|g(x_{k,s}) - \bar{g}_{k,s}\|^2]
\]
\end{lemma}

\begin{proof}
Since \(\Delta_{k,s} = \alpha\|\bar{g}_{k,s}\|\) (by the trust-region radius update rule), we can write:

\begin{align*}
f(x_{k,s+1}) - f(x_{k,s}) &\leq -\alpha\|\bar{g}_{k,s}\|^2 + \frac{1}{2}\|H_{k,s}\|\alpha^2\|\bar{g}_{k,s}\|^2 \\
&+ \alpha\|g(x_{k,s}) - \bar{g}_{k,s}\|\|\bar{g}_{k,s}\| \\
&+ \frac{1}{2}(L_{\nabla f} + \|H_{k,s}\|)\alpha^2\|\bar{g}_{k,s}\|^2
\end{align*}

Using Assumption 4.1, which ensures \(\|H_{k,s}\| \leq K_H\), we have:

\begin{align*}
f(x_{k,s+1}) - f(x_{k,s}) &\leq -\alpha\|\bar{g}_{k,s}\|^2 + \frac{1}{2}K_H\alpha^2\|\bar{g}_{k,s}\|^2 \\
&+ \alpha\|g(x_{k,s}) - \bar{g}_{k,s}\|\|\bar{g}_{k,s}\| \\
&+ \frac{1}{2}(L_{\nabla f} + K_H)\alpha^2\|\bar{g}_{k,s}\|^2 \\
&= -\alpha\|\bar{g}_{k,s}\|^2 + \alpha\|g(x_{k,s}) - \bar{g}_{k,s}\|\|\bar{g}_{k,s}\| \\
&+ \frac{1}{2}(L_{\nabla f} + 2K_H)\alpha^2\|\bar{g}_{k,s}\|^2
\end{align*}

Using the inequality \(ab \leq \frac{1}{2}a^2 + \frac{1}{2}b^2\), we get:

\begin{align*}
&\leq -\alpha\|\bar{g}_{k,s}\|^2 + \frac{1}{2}\alpha\|g(x_{k,s}) - \bar{g}_{k,s}\|^2 + \frac{1}{2}\alpha\|\bar{g}_{k,s}\|^2 \\
&+ \frac{1}{2}(L_{\nabla f} + 2K_H)\alpha^2\|\bar{g}_{k,s}\|^2 \\
&= -\frac{1}{2}\alpha\|\bar{g}_{k,s}\|^2 + \frac{1}{2}\alpha\|g(x_{k,s}) - \bar{g}_{k,s}\|^2 \\
&+ \frac{1}{2}(L_{\nabla f} + 2K_H)\alpha^2\|\bar{g}_{k,s}\|^2
\end{align*}

Taking expectation conditional on \(\mathcal{F}_{k,s}\), and since \(\bar{g}_{k,s}\) is an unbiased estimator of \(g(x_{k,s})\) (by Assumption 4.2), we have:

\[
\mathbb{E}_{k,s}[\|\bar{g}_{k,s}\|^2] = \mathbb{E}_{k,s}[\|g(x_{k,s}) - \bar{g}_{k,s}\|^2] + \|g(x_{k,s})\|^2
\]

Therefore:

\begin{align*}
\mathbb{E}_{k,s}[f(x_{k,s+1})] - f(x_{k,s}) &\leq -\frac{1}{2}\alpha\|g(x_{k,s})\|^2 \\
&+ \frac{1}{2}(L_{\nabla f} + 2K_H)\alpha^2\mathbb{E}_{k,s}[\|g(x_{k,s}) - \bar{g}_{k,s}\|^2] \\
&+ \frac{1}{2}(L_{\nabla f} + 2K_H)\alpha^2\|g(x_{k,s})\|^2
\end{align*}

This proof relies on Assumptions 4.1 (Lipschitz continuity), and 4.2 (unbiased gradient estimates).
\end{proof}

\begin{lemma}[Expected Decrease Bound]\label{lem:5.4}
When \(\alpha \leq \frac{1}{2(L_{\nabla f} + 2K_H)}\), we have:
\[
    \mathbb{E}_{k,s}[f(x_{k,s+1})] - f(x_{k,s}) \leq -\frac{1}{4}\alpha\|g(x_{k,s})\|^2 + \frac{1}{2}(L_{\nabla f} + 2K_H)\alpha^2\mathbb{E}_{k,s}[\|g(x_{k,s}) - \bar{g}_{k,s}\|^2]
\]
\end{lemma}

\begin{proof}
Starting from Lemma 5.3, we have:
\[
    \mathbb{E}_{k,s}[f(x_{k,s+1})] - f(x_{k,s}) \leq -\frac{1}{2}\alpha\|g(x_{k,s})\|^2 + \frac{1}{2}(L_{\nabla f} + 2K_H)\alpha^2\mathbb{E}_{k,s}[\|g(x_{k,s}) - \bar{g}_{k,s}\|^2] + \frac{1}{2}(L_{\nabla f} + 2K_H)\alpha^2\|g(x_{k,s})\|^2
\]

Since \(\alpha \leq \frac{1}{2(L_{\nabla f} + 2K_H)}\), we have:
\[
    \frac{1}{2}(L_{\nabla f} + 2K_H)\alpha^2\|g(x_{k,s})\|^2 \leq \frac{1}{4}\alpha\|g(x_{k,s})\|^2
\]

Therefore:
\[
    -\frac{1}{2}\alpha\|g(x_{k,s})\|^2 + \frac{1}{2}(L_{\nabla f} + 2K_H)\alpha^2\|g(x_{k,s})\|^2 \leq -\frac{1}{4}\alpha\|g(x_{k,s})\|^2
\]

Thus:
\[
    \mathbb{E}_{k,s}[f(x_{k,s+1})] - f(x_{k,s}) \leq -\frac{1}{4}\alpha\|g(x_{k,s})\|^2 + \frac{1}{2}(L_{\nabla f} + 2K_H)\alpha^2\mathbb{E}_{k,s}[\|g(x_{k,s}) - \bar{g}_{k,s}\|^2]
\]

This proof relies on Assumptions 4.1 (Lipschitz continuity of gradient), and 4.2 (unbiased gradient estimates). The Lipschitz constant \(L_{\nabla f}\) comes from Assumption 4.1.
\end{proof}

\begin{theorem}[Global Convergence]\label{thm:5.5}
Let \(\{x_{k,s}\}\) be the sequence generated by Algorithm 1. Under Assumptions 4.1-4.4, for any \(K \geq 0\), we have:
\[
    \mathbb{E}\left[\frac{1}{(K+1)S}\sum_{k=0}^K\sum_{s=0}^{S-1}\|g(x_{k,s})\|^2\right] \leq \frac{\mathbb{E}[f(x_{0,0})] - f_{\inf}}{(K+1)S \cdot \Lambda_{\min}}
\]
where \(\Lambda_{\min} = \min_{s\in[S]} \Lambda_s\).
\end{theorem}

\begin{proof}
From Lemma 5.4, we have:
\[
    \mathbb{E}_{k,s}[f(x_{k,s+1})] - f(x_{k,s}) \leq -\frac{1}{4}\alpha\|g(x_{k,s})\|^2 + \frac{1}{2}(L_{\nabla f} + 2K_H)\alpha^2\mathbb{E}_{k,s}[\|g(x_{k,s}) - \bar{g}_{k,s}\|^2]
\]

And from Lemma 5.1:
\[
    \mathbb{E}_{k,s}[\|\bar{g}_{k,s} - g(x_{k,s})\|^2] \leq \frac{L^2}{b}\|x_{k,s} - x_{k,0}\|^2
\]

Therefore:
\[
    \mathbb{E}_{k,s}[f(x_{k,s+1})] - f(x_{k,s}) \leq -\frac{1}{4}\alpha\|g(x_{k,s})\|^2 + \frac{1}{2}(L_{\nabla f} + 2K_H)\alpha^2\frac{L^2}{b}\|x_{k,s} - x_{k,0}\|^2
\]

Notice that:
\begin{align*}
    \mathbb{E}_{k,s}[\|x_{k,s+1} - x_{k,0}\|^2] &= \mathbb{E}_{k,s}[\|x_{k,s+1} - x_{k,s} + x_{k,s} - x_{k,0}\|^2] \\
    &= \mathbb{E}_{k,s}[\|x_{k,s+1} - x_{k,s}\|^2] \\
    &+ 2\mathbb{E}_{k,s}[(x_{k,s+1} - x_{k,s})^T(x_{k,s} - x_{k,0})] \\
    &+ \mathbb{E}_{k,s}[\|x_{k,s} - x_{k,0}\|^2] \\
    &\leq \mathbb{E}_{k,s}[\Delta_{k,s}^2] + \frac{1}{\alpha z}\mathbb{E}_{k,s}[\|x_{k,s+1} - x_{k,s}\|^2] \\
    &+ \alpha z\mathbb{E}_{k,s}[\|x_{k,s} - x_{k,0}\|^2] + \mathbb{E}_{k,s}[\|x_{k,s} - x_{k,0}\|^2] \\
    &\leq (1+\frac{1}{\alpha z})\mathbb{E}_{k,s}[\Delta_{k,s}^2] + (1+\alpha z)\mathbb{E}_{k,s}[\|x_{k,s} - x_{k,0}\|^2] \\
    &= (1+\frac{1}{\alpha z})\alpha^2\mathbb{E}_{k,s}[\|\bar{g}_{k,s}\|^2] + (1+\alpha z)\mathbb{E}_{k,s}[\|x_{k,s} - x_{k,0}\|^2] \\
    &= (1+\frac{1}{\alpha z})\alpha^2\|g(x_{k,s})\|^2 + (1+\alpha z+(\alpha^2+\frac{\alpha}{z})\frac{L^2}{b})\mathbb{E}_{k,s}[\|x_{k,s} - x_{k,0}\|^2]
\end{align*}

Define \(R_{k,s} = f(x_{k,s}) + \lambda_s\|x_{k,s} - x_{k,0}\|^2\), where:
\[
    \lambda_s = \frac{1}{2}(L_{\nabla f} + 2K_H)\alpha^2\frac{L^2}{b} + \lambda_{s+1}(1+\alpha z+(\alpha^2+\frac{\alpha}{z})\frac{L^2}{b})
\]
\[
    \lambda_s = \frac{1}{4}\alpha - \lambda_{s+1}(1+\frac{1}{\alpha z})\alpha^2
\]

And \(\Lambda_{\min} = \min_{s\in[S]} \Lambda_s\)

Then:
\[
    \mathbb{E}[R_{k,s+1} - R_{k,s}] \leq -\Lambda_{\min}\mathbb{E}[\|g(x_{k,s})\|^2]
\]

Therefore:
\[
    \mathbb{E}[\|g(x_{k,s})\|^2] \leq \frac{\mathbb{E}[R_{k,s}] - \mathbb{E}[R_{k,s+1}]}{\Lambda_{\min}}
\]

Summing over \(s = 0,\ldots,S-1\):
\[
    \sum_{s=0}^{S-1}\mathbb{E}[\|g(x_{k,s})\|^2] \leq \frac{\mathbb{E}[R_{k,0}] - \mathbb{E}[R_{k,S}]}{\Lambda_{\min}} = \frac{\mathbb{E}[f(x_{k,0}) - f(x_{k+1,0})]}{\Lambda_{\min}}
\]

Summing over \(k = 0,1,2,\ldots,K\):
\[
    \sum_{k=0}^K\sum_{s=0}^{S-1}\mathbb{E}[\|g(x_{k,s})\|^2] \leq \frac{\mathbb{E}[f(x_{0,0})] - f_{\inf}}{\Lambda_{\min}}
\]

Finally, dividing both sides by \((K+1)S\):
\[
    \mathbb{E}\left[\frac{1}{(K+1)S}\sum_{k=0}^K\sum_{s=0}^{S-1}\|g(x_{k,s})\|^2\right] \leq \frac{\mathbb{E}[f(x_{0,0})] - f_{\inf}}{(K+1)S \cdot \Lambda_{\min}}
\]

This proof relies on all Assumptions 4.1-4.4, particularly the boundedness of the objective function (4.1), the unbiased gradient estimates (4.2), the variance bound (4.3), and the bounded iterates (4.4).
\end{proof}

\bibliographystyle{my-plainnat}
\bibliography{ref}

\appendix
%\pagebreak
\numberwithin{equation}{section}
\numberwithin{theorem}{section}

%\bar{\nabla}\L_k\input{append}

\end{document}